\documentclass[11pt]{amsart}

\title{Some consequences of von Neumann algebra uniqueness}

\author{Thierry Giordano}

\address{Department of Mathematics and Statistics\\
University of Ottawa\\
585 King Edward Avenue\\
Ottawa, Ontario\\
KiN 6N5\\
Canada}

\email{giordano@uottawa.ca}

\author{Ping W. Ng}

\address{Department of Mathematics\\
University of Louisiana at Lafayette\\
217 Maxim D. Doucet Hall\\
P. O. Box 41010\\
Lafayette, Louisiana\\
70504-1010\\
USA}

\email{png@louisiana.edu}

\thanks{T.G. was
partially supported by a grant from NSERC Canada.}

\newtheorem{thm}{Theorem}[section]
\newtheorem{prop}[thm]{Proposition}
\newtheorem{lem}[thm]{Lemma}
\newtheorem{cor}[thm]{Corollary}
\newtheorem{df}{Definition}[section]

\newcommand{\A}{\mathcal{A}}
\newcommand{\B}{\mathcal{B}}
\newcommand{\C}{\mathcal{C}}
\newcommand{\F}{\mathcal{F}}
\newcommand{\K}{\mathcal{K}}
\newcommand{\G}{\mathcal{G}}
\newcommand{\Mul}{\mathcal{M}}

\newcommand{\N}{\mathcal{N}}

\newcommand{\h}{\mathcal{H}}

\newcommand{\Fr}{\mathbb{F}}

\def \bib(#1;#2;#3;#4;#5;#6) {{#1}, {\it #2} {#3},
{\bf#4} (#5) {#6}\par\smallskip}

\date{\today}

\numberwithin{equation}{section}

\begin{document}

\maketitle

\begin{abstract}
In this note, we derive some consequences of 
the von Neumann algebra uniqueness theorems developed in the
previous paper \cite{CGNN}.  
In particular, 
\begin{enumerate}
\item we solve
a question raised in \cite{FutamuraKataokaKishimoto}, by proving that if $\A$ is a separable simple nuclear C*-algebra and $\pi_i, \, i = 1,2,$ are type III representations of $A$ on a separable Hilbert space, then for $\pi_1$ and $\pi_2$ being algebraically equivalent, it is necessary and sufficient that there is an automorphism $\alpha$ of $A$ such that $\pi_1\circ \alpha$ and $\pi_2$ are quasi-equivalent.  
\item we give a new (short)
proof of the equivalence of injectivity and extreme amenability (of the 
corresponding unitary group)
for countably decomposable
properly infinite von Neumann algebras.
\item using ideas of \cite{Pestov}, we show that the Connes embedding 
problem is equivalent
to many topological groups having the Kirchberg property.  
\end{enumerate}
\end{abstract}

\section{Introduction}

To prove in his famous 1967 paper that the now called Powers factors $R_{\lambda}, \, 0 < \lambda <1, $ are non isomorphic, R.T. Powers \cite{Powers} proved that if $A$ is a UHF C*-algebra and $\pi_i, \, i = 1,2,$ are representations of $A$ on a separable Hilbert space, then for $\pi_1$ and $\pi_2$ being algebraically equivalent (i.e., the von Neumann algebras generated by $\pi_1(A)$ and $\pi_2(A)$ are isomorphic), it is necessary and sufficient that there is an automorphism $\alpha$ of $A$ such that $\pi_1\circ \alpha$ and $\pi_2$ are quasi-equivalent.  
\smallskip

Powers' characterization was generalized either to a larger class of C*-algebras $A$ or to special classes of representations (see \cite{Bratteli}, \cite{KishimotoOzawaSakai}, {\cite{KishimotoFlows}, \cite{KishimotoInternat}}). In the first section of this paper, we resolve a question raised in \cite{FutamuraKataokaKishimoto} by showing that 

\begin{thm} Let $\A$ be a separable simple nuclear C*-algebra,
and let $\pi_1$ and $\pi_2$ be nondegenerate type III representations
of $\A$ on a separable Hilbert space.
Then $\pi_1$ and $\pi_2$ are algebraically equivalent if and only if
there exists $\alpha \in AInn(\A)$ such that 
$\pi_1 \circ \alpha$ and $\pi_2$ are quasi-equivalent.
\end{thm}

The main step of the proof of this theorem is a consequence of \cite{CGNN} Corollary
3.5.
\medskip


In Section 3, we use Voiculescu-type uniqueness theorems developed in  \cite{CGNN} to study extreme amenability. Recall that a topological group $G$ is said to be {\it extremely  amenable} if every continuous action of $G$, on a compact Hausdorff topological space, has a fixed point and that that no locally compact group is extremely amenable \cite{Veech}. However, large classes of interesting ``massive" (i.e., non locally compact) topological groups have been shown to be extremly amenable. A basic example is Gromov and Milman's result that for a separable infinite dimensional Hilbert space $\h$, the unitary group $U(\h)$, given the weak* topology, is extremely amenable \cite{GromovMilman}. Subsequently it was shown in \cite{GP2} that many naturally occuring nonlocally compact groups - coming from operator algebras, dynamical systems and combinatorics are extremely amenable. We give in Theorem \ref{thm:extremeamenability} another proof of the following result, originally proven in \cite{GP2}. 

\begin{thm}  Let $\Mul$ be a countably decomposable
properly infinite injective von Neumann
algebra.
Let $U(\Mul)$ be the unitary group of $\Mul$, given the $\sigma$-strong*
topology.

Then $\Mul$ is injective if and only if $U(\Mul)$ is extremely
amenable.
\end{thm}
\medskip

Finally we use our techniques to study the Kirchberg property and the Connes embedding problem. The definitions are recalled in Section 4.  In \cite{Pestov}, Pestov and Uspenskij observed that an affirmative answer to the Connes embedding problem is equivalent to $U(\h)$ having the Kirchberg property.
We generalize their result by showing that:

\begin{thm}  The following statements are equivalent:
\begin{enumerate}
\item The Connes embedding problem has an affirmative answer.

\item For every unital separable simple nuclear $C^*$-algebra $\A$,
$U( \Mul( \A \otimes \K))$ has the Kirchberg property.

\item There exists a unital separable simple nuclear $C^*$-algebra $\A$
such that $U(\Mul(\A \otimes \K))$ has the Kirchberg property.
\item For every countably decomposable
injective properly infinite von Neumann algebra $\Mul$,  
$U( \Mul )$ has the Kirchberg property.
\item There exists an injective properly infinite
von Neumann algebra $\Mul$ with separable predual 
such that $U( \Mul )$ has the Kirchberg property.
\end{enumerate}
\end{thm}


\section{Quasi-equivalent representations}

In \cite{FutamuraKataokaKishimoto},  Futamura, Kataoka and Kishimoto prove that if $A$ belongs to to a restricted class of simple, separable, nuclear C*-algebras, and if $\pi_1$ and $\pi_2$ are  type $III$ representations of $A$ on a separable Hilbert space, then they are algebraically equivalent if and only if there is an asymptotically inner automorphism $\alpha$ of $A$ such that $\pi_1\circ \alpha$ and $\pi_2$ are quasi-equivalent. 

 In this section, we show, by modifying our previous results, that Futamura, Kataoka and Kishimoto's statement is true for all  simple, separable, nuclear C*-algebras. Our result also generalizes
 results of Powers  \cite{Powers}, of Bratteli \cite{Bratteli}, of Kishimoto, Ozawa, and Sakai \cite{KishimotoOzawaSakai} and of Kishimoto 
 {\cite{KishimotoFlows}, \cite{KishimotoInternat}}. As indicated in the introduction, Powers 
\cite{Powers} was interested in constructing a continuum of nonisormorphic
type III factors, and Bratteli \cite{Bratteli} in 
representations of AF-algebras, as well as the problem of when 
the pure state space is homogeneous (see also 
\cite{KishimotoOzawaSakai}).  Kishimoto \cite{KishimotoFlows} 
was interested in representations
that were covariant with respect to an approximately inner flow.  
We refer the reader to these papers and their references for details
as well as basic definitions and background results.

The following proposition is an immediate consequence of \cite{CGNN} Corollary
3.5 and is the starting point of our proof. Before stating it, let us recall the notion of a 
{\it full} *-homomorphism:

\begin{df} Let $\A$ and $\B$ be two C*-algebras. Then a *-homomorphism  $\phi : \A \rightarrow \B$ is said to be full if for every nonzero positive element $a \in \A$, its image $\phi(a)$ is a full element of $\B$ (i.e., $\phi(a)$ is not contained in any proper C*-algebra ideal of $\B$). 
\label{df:Full homomorphism}
\end{df} 

Note that a full *-homomorphism is necessarily injective and that conversely  an injective *-homomorphism from a C*-algebra to a type III and countably decomposable von Neumann factor is full.

\begin{prop} Let $\Mul$ be a countably decomposable properly infinite
von Neumann algebra, and 
let  $\A$ be a nuclear C*-algebra.
Let $\phi, \psi : \A \rightarrow \Mul$ be two full *-homomorphisms.

Then there exists a net $\{ v_{\alpha} \}$ of partial isometries 
in $\Mul$ such that 
for all $a \in \A$,
$$\| v_{\alpha}^* \phi(a) v_{\alpha} - \psi(a) \| \rightarrow 0$$ 
\label{prop:NormPartialIsometryUniqueness}
\end{prop}

The next three lemmas generalize arguments from  
\cite{DavidsonBook} Corollary II.5.4, II.5.5 and II.5.6.  

\begin{lem}  Let $\Mul$ be a countably decomposable properly infinite
von Neumann algebra, and
let  $\A$ be a full nuclear C*-subalgebra of $\Mul$ with $1_{\Mul} \in \A$.
Let $\rho : \A \rightarrow \Mul$ be a full $*$-homomorphism such 
that $\rho(1_{\Mul}) = 1_{\Mul}$.   

Then there exists a net $\{ v_{\alpha} \}$ of isometries of $\Mul$
such that for all $a \in \A$.
$$\| v_{\alpha} \rho(a) -  a v_{\alpha} \| \rightarrow 0$$
\label{lem:SomeIsometriesForUniquenessSpecialLemma} 
\end{lem}

\begin{proof}

    By Proposition \ref{prop:NormPartialIsometryUniqueness}, 
there exists a net $\{ v_{\alpha} \}$ of partial isometries in $\Mul$
such that for all $a \in \A$,  
\begin{equation}
\| \rho(a) - v_{\alpha}^* a v_{\alpha} \| \rightarrow 0
\label{equ:PartialIsometryUniqueness}
\end{equation}

    Note that since $1_{\Mul} \in \A$ and 
$\rho(1_{\Mul}) = 1_{\Mul}$, 
equation (\ref{equ:PartialIsometryUniqueness}) implies
that for every $\epsilon > 0$,
$\| 1_{\Mul} - v_{\alpha}^* v_{\alpha} \| < \epsilon$.
for sufficiently ``large" $\alpha$.  Hence,
we may assume that for all $\alpha$, 
$v_{\alpha}$ is an isometry in $\Mul$.
 
    Hence, for all $a \in \A$,
\begin{eqnarray*}
& & (v_{\alpha} \rho(a) -  a v_{\alpha})^* ( v_{\alpha} \rho(a) 
-  a v_{\alpha}) \\
& = & (\rho(a^*) v_{\alpha}^* - v_{\alpha}^* a^*) ( v_{\alpha} \rho(a)
-  a v_{\alpha}) \\
& = &  \rho(a^*) \rho(a) - \rho(a^*) v_{\alpha}^* a v_{\alpha} 
- v_{\alpha}^* a^* 
v_{\alpha} \rho(a) + v_{\alpha}^* a^* a v_{\alpha} \\
& = & \rho(a^*)(\rho(a) - v_{\alpha}^* a v_{\alpha}) + 
(\rho(a^*) - v_{\alpha}^* a^* v_{\alpha}) \rho(a) + 
(v_{\alpha}^* a^* a v_{\alpha} - \rho(a^* a) ) \\
& \rightarrow & 0  
\end{eqnarray*}
\end{proof}

Let $\A$ be a C*-algebra, and let $\Mul$ be a von Neumann algebra. 
We say that a *-homomorphism 
$\rho : \A \rightarrow \Mul$  is \emph{nondegenerate} if
$1_{\overline{\rho(\A)}^{weak*}} = 1_{\Mul}$.
If $\A$ is a C*-subalgebra of $\Mul$ then 
$\A$ is a \emph{nondegenerate} C*-subalgebra if the inclusion
map $\A \rightarrow \Mul$ is nondegenerate. 

\begin{lem} Let $\Mul$ be a countably decomposable properly
infinite von Neumann algebra.
Let $\A \subseteq \Mul$ be a
full nondegenerate nuclear C*-subalgebra
and let $\rho : \A \rightarrow \Mul$ be a full nondegenerate *-homomorphism. 

Let $S_1, S_2 \in \Mul$ be elements such that 
$S_1^* S_1 = S_2^* S_2 = 1_{\Mul}$ and $S_1 S_1^* + S_2 S_2^* = 1_{\Mul}$,
and consider 
the full nondegenerate *-homomorphism
$$\A \rightarrow \Mul : a \mapsto S_1 a S_1^* + S_2 \rho(a) S_2^*.$$  

Then for every $\epsilon > 0$, for every finite subset $\F \subset \Mul$,
there exists a unitary $w \in \Mul$ such that 
$$\| w (S_1 a S_1^* + S_2 \rho(a) S_2^*)w^* - a \| < \epsilon$$
for all $a \in \F$. 
\label{lem:AnotherSpecialLemmaButWithRho}
\end{lem}

\begin{proof}
  Firstly, we may assume that $\A$ is unital (i.e., $1_{\Mul} \in \A$) and
$\rho(1_{\Mul}) = 1_{\Mul}$.  (For otherwise, 
we can replace $\A$ with its unitization
$\A + \mathbb{C}1_{\Mul}$ (which will still be nuclear) and
replace $\rho$ by the  (unitized) map
$a + \alpha 1_{\Mul} \mapsto \rho(a) + \alpha 1_{\Mul}$.  
Here is a rough sketch of the proof that the unitizations are 
still full:  It suffices to prove that if $a \in \A$ is nonzero self-adjoint
and $\alpha \ge 0$ is such that $a + \alpha 1 \in (\A + \mathbb{C}1)_+$, then
$\rho(a) + \alpha 1$ is a full element of $\Mul$.  To do this, find a
nonzero element $c \in \A$ with $0 \leq c \leq a +\alpha1$.
Then $\rho(c)$ is a full element of $\Mul$ and $0 \leq \rho(c) \leq \rho(a) +
\alpha 1$.  Hence, $\rho(a) + \alpha 1$ is a full element of $\Mul$.)

Let $\epsilon > 0$ and a finite subset $\F \subset \A$ be given.
We may assume that $\F$ is self-adjoint; i.e., for all $a \in \A$,
if $a \in \F$ then $a^* \in \F$.

Let $\lambda : \A \rightarrow \Mul$ be given by 
$\lambda(a) =_{df} S_2 \rho(a) S_2^*$ for all $a \in \A$, and
let $q =_{df} \lambda(1) = S_2 S_2^*$.

Let $\{ j_n \}_{n=1}^{\infty}$ be a sequence of partial isometries in $\Mul$
such that
$j_n^* j_n = q$ for all $n \geq 1$, $j_1 = j_1^* = q$, and 
$\sum_{n=1}^{\infty} j_n j_n^* = 1$.

Set $T =_{df} \sum_{n=1}^{\infty} j_n j_{n+1}^* \in \Mul$.
Then
$$T^* T = \sum_{n, m} j_{n+1} j_n^* j_m j_{m+1}^* = 
\sum_{n=1}^{\infty} j_{n+1} j_n^* j_n j_{n+1}^* = 1 - q$$
and 
$$T T^* = \sum_{n,m} j_n j_{n+1}^* j_{m+1} j_m^* = 
\sum_{n=1}^{\infty} j_n j_{n+1}^* j_{n+1} j_n^* = 1.$$  
(Hence, $T^*$ is an isometry.)

Set $\lambda^{(\infty)} : \A \rightarrow \Mul$ by
$\lambda^{(\infty)}(a) =_{df} \sum_{n=1}^{\infty} j_n \lambda(a) j_n^*$
for all $a \in \A$.
Hence, 
$\lambda^{(\infty)} (1) = \sum_{n=1}^{\infty} j_n q j_n^* = 1$.
Hence, $\lambda^{(\infty)}$ is a full unital *-homomorphism.

    By 
Lemma \ref{lem:SomeIsometriesForUniquenessSpecialLemma},  
let $v \in \Mul$ be an isometry such that 
\begin{equation}
\| v \lambda^{(\infty)}(a) - a v \| < \epsilon/5 
\label{equ:AnApplicationOfSomeIsometriesSpecialLemma}
\end{equation}
for all $a \in \F$.
 
Now let 
$$w =_{df} (1 - v v^* + v T^* v^*)T + v j_1 \in \Mul.$$
One can check that $w$ is a unitary of $\Mul$.

Note that since $T^*$ is an isometry with $T T^* = 1- q$, 
we may assume that $S_1 = T^* = \sum_{n=1}^{\infty} j_{n+1} j_n^*$.

Hence, one can check that for all $a \in \A$,  
\begin{eqnarray*}
&  & a w - w (S_1 a S_1^* + S_2 \rho(a) S_2^*)\\ 
&  = &  a w - w (T^* a T + \lambda(a)) \\
& = &  (v v^* a - a v v^*) T + ( a v T^* v^* - v T^* v^* a)T + (a v j_1 -
v j_1 \lambda(a)) 
\end{eqnarray*}

Also, one can check that for all $a \in \A$, we have the following:
\begin{enumerate}
\item $(v v^* a - a v v^*) T =  [ v (a^* v - v \lambda^{(\infty)}(a^*))^*
-(av - v \lambda^{(\infty)}(a))v^*]T$.   
\item $(av T^* v^* - v T^* v^* a )T = [(av - v \lambda^{(\infty)}(a))T^* v^*
- v T^* (a^* v - v \lambda^{(\infty)}(a^*))^*]T$\\
(Note that $T^* \lambda^{(\infty)}(a) = \lambda^{(\infty)}(a) T^*$ for all
$a \in \A$.) 
\item $a v j_1 - v j_1 \lambda(a) = (a v - v \lambda^{(\infty)}(a)) j_1$ 
\end{enumerate}
 
   From the above and from 
(\ref{equ:AnApplicationOfSomeIsometriesSpecialLemma}),
we have that for all $a \in \F$,
$$\| aw - w(S_1 a S_1^* + S_2 \rho(a) S_2^*) \| < \epsilon.$$

Hence, for all $a \in \F$,
$$\| a - w(S_1 a S_1^* + S_2 \rho(a) S_2^*) w^* \| < \epsilon$$
as required.
\end{proof}

\begin{lem}  Let $\Mul$ be a countably decomposable properly infinite
von Neumann algebra and let $\A$ be a nuclear C*-algebra.
Let $\phi, \psi : \A \rightarrow \Mul$ be full nondegenerate *-homomorphisms,
and let $S_1, S_2 \in \Mul$ be elements with
$S_i^* S_i = 1$ ($i = 1,2$) and
$S_1 S_1^* + S_2 S_2^* = 1$.

Then $\psi$ and $Ad(S_1)\psi + Ad(S_2)\phi$ are (norm-) approximately
unitarily equivalent. I.e., 
there exists a net $\{ u_{\alpha} \}$ of unitaries in 
$\Mul$ such that for all $a \in \A$,
$$\| u_{\alpha}^* \psi(a) u_{\alpha} - (S_1 \psi(a) S_1^* + 
S_2 \phi(a) S_2^* ) \| \rightarrow 0$$
\label{lem:StillAnotherSpecialLemmaButThisTimeNoRho}
\end{lem}

\begin{proof}
Since $\psi$ is full and nondegenerate, we can identify $\A$ with $\psi(\A)
\subset \Mul$, view $\A$ as a full nondegenerate C*-subalgebra of $\Mul$, and 
replace $\psi$ with the inclusion map.
The result then follows from Lemma \ref{lem:AnotherSpecialLemmaButWithRho}
(where we take $\rho = \phi$). 
\end{proof}

\begin{prop}
Let $\Mul$ be a countably decomposable properly infinite von Neumann algebra 
and let $\A$ be a nuclear C*-algebra.
Let $\phi, \psi : \A \rightarrow \Mul$ be full nondegenerate
*-homomorphisms.

Then $\phi$ and $\psi$ are approximately unitarily equivalent in norm.
I.e., there exists a net $\{ u_{\alpha} \}$ of unitaries in $\Mul$
such that for all $a \in \A$,
$$\| u_{\alpha}^* \phi(a) u_{\alpha} - \psi(a) \| \rightarrow 0$$ 
\label{prop:NormAUEProperlyInfiniteCase}
\end{prop}

\begin{proof}
Let $S_1, S_2 \in \Mul$ be isometries with $S_1 S_1^* + S_2 S_2^* = 1$.   
By applying Lemma
\ref{lem:StillAnotherSpecialLemmaButThisTimeNoRho} twice,
we have that $\phi$ and $Ad(S_1) \psi + Ad(S_2) \phi$
are (norm-) approximately unitarily equivalent, and 
$\psi$ and $Ad(S_1) \psi + Ad(S_2) \phi$ are (norm-) approximately
unitarily equivalent.
Hence, $\phi$ and $\psi$ are (norm-) approximately unitarily equivalent,
as required. 
\end{proof}

As a corollary of Proposition 
\ref{prop:NormAUEProperlyInfiniteCase}, we get the following: 

\begin{cor}  Let $\Mul$ be a countably generated von Neumann, and let
$\A$ be a nontype I simple nuclear C*-algebra.
Suppose that $\phi, \psi : \A \rightarrow \Mul$
are  nondegenerate *-homomorphisms.
Suppose also that either
\begin{enumerate}
\item $\A$ is purely infinite, or 
\item $\Mul$ has no type II summand, or
\item $\A$ is unital and $\Mul$ has no type $II_1$ summand.
\end{enumerate}
Then $\phi$ and $\psi$ are both (norm-) approximately unitarily equivalent.
I.e., there exists a net $\{ u_{\alpha} \}$ of unitaries in $\Mul$
such that for all $a \in \A$,
$$\| u_{\alpha}^* \phi(a) u_{\alpha} - \psi(a) \| \rightarrow 0$$
\label{cor:NormAUE}
\end{cor}

\begin{proof}
Firstly, since $\Mul$ is countably generated, it is a (possibly
infinite) direct product of
von Neumann algebras with separable predual.  
Hence, we may assume that $\Mul$ has separable predual.

Decompose $\Mul$ into a direct sum 
$\Mul = \Mul_{\infty} \oplus \Mul_{II_1} \oplus \Mul_{I_f}$
where $\Mul_{\infty}$ is properly infinite, 
$\Mul_{II_1}$ is type $II_1$ and $\Mul_{I_f}$ is finite type I.

Since $\A$ is nontype I and simple, and since 
we have a nondegenerate *-homomorphism $\phi : \A \rightarrow \Mul$,
$\Mul_{I_f} = 0$. Hence,
$$\Mul = \Mul_{\infty} \oplus \Mul_{II_1}.$$

Let us first assume (1).   
Since $\A$ is simple purely infinite and since $\phi : \A \rightarrow \Mul$
is a  nondegenerate *-homomorphism, $\Mul_{II_1} = 0$.
Hence, $\Mul = \Mul_{\infty}$; i.e., $\Mul$ is properly infinite.
Since $\A$ is simple purely infinite, let $p \in \A$ be a nonzero properly
infinite projection. 
Hence, $\phi(p)$ and $\psi(p)$ are properly infinite projections
in $\Mul$. Since $\A$ is simple and $\phi, \psi$ are nondegenerate,
we have, by \cite{KadisonVolume2} Corollary 6.3.5, that 
$\phi(p)$ and $\psi(p)$ are both Murray-von Neumann equivalent
to $1_{\Mul}$.
Hence, since $\A$ is simple, $\phi$ and $\psi$ are full. Hence,
by Proposition \ref{prop:NormAUEProperlyInfiniteCase}, 
$\phi$ and $\psi$ are (norm-) approximately unitarily equivalent.

Next, assume (2).   
Since $\Mul$ has not type $II$ summand,
$\Mul_{II_1} = 0$ and $\Mul = \Mul_{\infty}$ is properly infinite.
Indeed, $\Mul$ is a type III von Neumann algebra with separable predual.
Let $a, b \in \A$ be nonzero positive elements such that 
$0 \leq a \leq b$ and $ab = a$.
Hence, there exist (necessarily nonzero) projections $p, q \in \Mul$
such that $\phi(a) \leq p \leq \phi(b)$
and $\psi(a) \leq q \leq \psi(b)$. (E.g., take $p, q$ to be the support
projections of $\phi(a), \psi(a)$ respectively.)     
Since $\phi(a) \neq 0$, $\psi(a) \neq 0$, and since $\Mul$ is type III,
$p, q$ must be properly infinite projections in $\Mul$ (e.g., see
\cite{KadisonVolume2} Proposition 6.3.7).   
Since $\phi, \psi$ are nondegenerate and since
$\phi(a) \leq p$ and $\psi(a) \leq q$, 
$p, q$ both must have central carrier $1$.  Hence,
by \cite{KadisonVolume2} Corollary 6.3.5, $p, q$ are both Murray--von Neumann
equivalent to $1$. 
Hence, since $p \leq \phi(b)$ and $q \leq \psi(b)$, and since 
$\A$ is simple, $\phi, \psi$ are full *-homomorphisms.
Hence, by Proposition \ref{prop:NormAUEProperlyInfiniteCase},    
$\phi$ and $\psi$ are (norm-) approximately unitarily equivalent, 
as required.

Finally, assume (2).  Since $\Mul$ does not contain a type $II_1$ direct
summand, $\Mul = \Mul_{\infty}$; i.e., $\Mul$ is properly infinite.  
Since $\A$ is unital and since $\phi, \psi$ are nondegenerate,
$\phi(1_{\A}) = \psi(1_{\A}) = 1_{\Mul}$.
Hence, since $\A$ is simple, $\phi, \psi$ are full *-homomorphisms.
Hence, by Proposition \ref{prop:NormAUEProperlyInfiniteCase},  
$\phi$ and $\psi$ are (norm-) approximately unitarily equivalent.
This completes the final case. 
\end{proof}

Using Corollary \ref{cor:NormAUE}, we have the following result,
which generalizes \cite{FutamuraKataokaKishimoto} Corollaries 3.6 and
3.8.

\begin{cor} Let $\A$ be a simple separable nuclear C*-algebra,
and let $\pi_1$ and $\pi_2$ be nondegenerate 
representations of $\A$ on a separable
Hilbert space $\h$ such that such that 
$\pi_1(\A)'' = \pi_2(\A)'' = \Mul$.
Suppose that either
\begin{enumerate}
\item $\A$ is purely infinite or 
\item $\Mul$ does not contain a type II direct summand or 
\item $\A$ is unital and $\Mul$ does not contain a type $II_1$
direct summand.
\end{enumerate}
Then there exists a sequence $\{ u_n \}$ of unitaries in $\Mul$
such that for all $a \in \A$,
$$\| u_n^* \pi_1(a) u_n - \pi_2(a) \| \rightarrow 0$$
\label{cor:FirstGeneralization}
\end{cor}

\begin{proof}
 If $\A$ is isomorphic to the algebra of compact operators on a separable
Hilbert space, then the result follows from \cite{DavidsonBook}  
Corollary I.10.6 and Theorem I.10.7.   
Hence, we may assume that $\A$ is not type I. 

If $\A$ is nontype I then the result follows from
Corollary \ref{cor:NormAUE}.
\end{proof}

Let $\A$ be a C*-algebra.
Let $\pi_1$ and $\pi_2$ be nondegenerate representations of $\A$
on Hilbert spaces $\h_1$ and $\h_2$ respectively.
Recall that $\pi_1$ and $\pi_2$ are said to be \emph{algebraically 
equivalent} if $\pi_1(\A)''$ and $\pi_2(\A)''$ are isomorphic as
von Neumann algebras.  Recall that $\pi_1$ and $\pi_2$ are said to 
be \emph{quasi-equivalent} if the map $\pi_1(a) \mapsto \pi_2(a)$
(for $a \in \A$) extends to an isomorphism of $\pi_1(\A)''$ onto
$\pi_2(\A)''$.

Next, recall that a *-automorphism $\alpha \in Aut(\A)$ is said to be
\emph{asymptotically inner} if there is a continuous
path $\{ u_t \}_{t \in [0, \infty)}$ of unitaries in 
$\A$ (or $\A + \mathbb{C} 1$ if $\A$ is nonunital) such that 
$lim_{t \rightarrow \infty} u_t^* a u_t = \alpha(a)$ for all $a \in \A$.
We let $AInn(\A)$ denote the group of asymptotically inner automorphisms
of $\A$.

The next result generalizes \cite{FutamuraKataokaKishimoto}
Corollary 4.3 and part of Corollary 4.4. 

\begin{thm} Let $\A$ be simple separable nuclear C*-algebra.
Let $\pi_1$ and $\pi_2$ be nondegenerate representations of $\A$
such that 
$\pi_1(\A)''$ (or $\pi_2(\A)''$) has separable predual.
Suppose that one of the following conditions hold:
\begin{enumerate}
\item $\A$ is purely infinite.
\item $\pi_1(\A)''$ (or $\pi_2(\A)''$) contains no type II summand
\item $\A$ is unital and $\pi_1(\A)''$ (or $\pi_2(\A)''$) does not
contain a type $II_1$ summand. 
\end{enumerate}

Then $\pi_1$ and $\pi_2$ are algebraically equivalent if and only if
there exists an $\alpha \in AInn(\A)$ such that $\pi_1 \circ \alpha$
and $\pi_2$ are quasi-equivalent.
\label{thm:SecondGeneralization}
\end{thm}

\begin{proof}
 The ``if'' direction is clear. 

 We now prove the ``only if'' direction.  

Case 1: 
$\pi_1(\A)'' = \pi_2(\A)''$.

In this case, the result follows from Corollary \ref{cor:FirstGeneralization}
and \cite{FutamuraKataokaKishimoto} Theorem 4.1. 

Case 2:  General case.

Let $\Phi : \pi_1(\A)'' \rightarrow \pi_2(\A)''$ be an isomorphism.
Consider the nondegenerate representation
$\Phi \circ \pi_1$ of $\A$. 
Since $\Phi$ is normal and by hypothesis,
$(\Phi \circ \pi_1)(\A)'' = \Phi( \pi_1(\A)'') = \pi_2(\A)''$. 
Hence, by Case 1, 
there exists 
$\alpha \in AInn(\A)$ such that 
$\Phi \circ \pi_1 \circ \alpha$ and  $\pi_2$ are quasi-equivalent.
Since $\pi_1 \circ \alpha$ and $\Phi \circ \pi_1 \circ \alpha$
are quasi-equivalent, 
$\pi_1 \circ \alpha$ and $\pi_2$ are quasi-equivalent.
\end{proof}

As a consequence of the above result, we resolve a question in 
\cite{FutamuraKataokaKishimoto}:

\begin{thm} Let $\A$ be a separable simple nuclear C*-algebra,
and let $\pi_1$ and $\pi_2$ be nondegenerate type III representations
of $\A$ on a separable Hilbert space.
Then $\pi_1$ and $\pi_2$ are algebraically equivalent if and only if
there exists $\alpha \in AInn(\A)$ such that 
$\pi_1 \circ \alpha$ and $\pi_2$ are quasi-equivalent.
\end{thm}

\section{Extreme amenability}

A (Hausdorff) topological group $G$ is said to 
be \emph{extremely amenable}
if every continuous action of $G$ on a compact Hausdorff space
has a fixed point.

 If $\Mul$ is an injective  
von Neumann algebra with no type I summand
then its unitary group $U(\Mul)$ 
with the $\sigma$-strong* topology
is extremely amenable (see \cite{GP2} Theorem 3.3 and Corollary 3.6).
We refer the reader to \cite{GP2} for the background and basic 
definitions 
in the subject.  


  In this section, if $\Mul$ is a
countably decomposable properly infinite injective von Neumann algebra,
we use our uniqueness results from \cite{CGNN}  
to given another proof of the extreme amenability of $U(\Mul)$.

  Let $\A, \C$ be C*-algebras. Then a *-homomorphism
$\rho : \A \rightarrow \C$ is said to be \emph{full} if 
for every nonzero positive element $a \in \A$,
$\rho(a)$ is a full element of $\C$. If $\A$ is a C*-subalgebra of $\C$, 
then $\A$ is said to be a \emph{full C*-subalgebra} of $\C$ if
the inclusion map $\A \subseteq \C$ is full.    
  
  We will be using the following result, which is a 
consequence of \cite{CGNN} Proposition 4.3 and Lemma 3.9. 

\begin{prop} Let $\Mul$ be a countably decomposable properly infinite
injective von Neumann algebra, and let $\A$ be a unital C*-algebra.
Suppose that 
$\phi, \psi : \A \rightarrow \Mul$ are two unital
full *-homomorphisms.

Then there exists a net $\{ u_{\alpha} \}$ of unitaries in $\Mul$
such that for all $a \in \A$, 
$$u_{\alpha} \phi(a) u_{\alpha}^*  \rightarrow \psi(a)$$
in the $\sigma$-strong* topology.
\label{prop:VNAUniqueness}
\end{prop}



Following the arguments of \cite{KucerovskyNg}, we will use the notion
of a locally dense net. 

\begin{df}  
Let $G$ be a topological group, and let $\Pi$ be 
a uniform structure on $G$ that
is compatible with the topology on $G$. 
Let us say that a net of subgroups $\{ H_{\lambda} \}_{\lambda \in
\Lambda}$ of $G$ is \emph{locally dense} in $G$ if for every entourage
$\mathcal{E} \in \Pi$, for every finite set
$\{ g_i \}_{i=1}^{n}$  in $G$, there is a $\lambda_0 \in \Lambda$ such that
for $\lambda \geq \lambda_0$ and for $1 \leq i \leq n$, 
$(\{ g_i \} \times H_{\lambda}) \cap \mathcal{E} \neq \phi$.  
\end{df}


The next lemma is an exercise in point set topology.

\begin{lem}  Let $G$ be a topological group, and let $\Pi$ be 
a uniform structure on $G$ that is compatible with the topology on $G$.  
If $(G, \Pi)$ has
a locally dense net $\{ H_{\lambda} \}_{\lambda \in \Lambda}$
consisting of extremely amenable subgroups, then $G$ is extremely 
amenable.
\label{lem:pointwise}
\end{lem}


   We now fix a notation.  For a von Neumann algebra $\Mul$ and
a normal positive linear functional $\chi \in \Mul_*$, 
let $\| . \|^{\sharp}_{\chi}$ be the seminorm on $\Mul$ that is defined
by $\| x \|^{\sharp}_{\chi} =_{df} \sqrt{\chi(x^*x) + \chi(x x^*)}$ for all
$x \in \Mul$.  

\begin{thm}  Let $\Mul$ be a countably decomposable
properly infinite injective von Neumann
algebra.
Let $U(\Mul)$ be the unitary group of $\Mul$, given the $\sigma$-strong*
topology.

Then $\Mul$ is injective if and only if $U(\Mul)$ is extremely
amenable.
\label{thm:extremeamenability}
\end{thm}

\begin{proof}

 The ``if'' direction follows from \cite{delaHarpe}  and the
 fact that extreme amenability implies amenability.

 We now prove the ``only if'' direction.
  
  Since $\Mul$ is countably decomposable, let $\chi \in \Mul_*$ be
a faithful normal state on $\Mul$.  Then the $\sigma$-strong* topology
on $U(\Mul)$ is given by the metric induced by the norm
$\| . \|^{\sharp}_{\chi}$ (which gives a uniform structure 
on $U(\Mul)$).

Let $\epsilon > 0$ be given and let $x_1, x_2, ..., x_m$
be a finite subset of $U ( \Mul )$.

Now since $\Mul$ is properly infinite, let
$\N$ be a von Neumann algebra and $\h$ a separable infinite dimensional
Hilbert space such that $\Mul \cong \N \otimes \mathbb{B}( \h)$.

Let $\A$ be the unital C*-subalgebra of $\Mul$ that is generated by
$\{ x_1, x_2, ..., x_m \}$.

Now let $\psi : \A \rightarrow 1_{\N} \otimes \mathbb{B}(\h) \subset \Mul$
be a full unital *-homomorphism.
Since $\Mul$ is properly infinite, let $S_1, S_2 \in \Mul$ 
be isometries with
orthogonal ranges with $S_1 S_1^* + S_2 S_2^* = 1_{\Mul}$ 
such that 
$$\| x_k - (S_1 x_k S_1^* + S_2 \psi(x_k) S_2^*) \|^{\sharp}_{\chi}
< \epsilon/3$$  
for $1 \leq k \leq m$

Let $\phi : \A \rightarrow \Mul$ be the full unital *-homomorphism
that is given by $\phi(a) =_{df} S_1 a S_1^* + S_2 \psi(a) S_2^*$
for all $a \in \A$.
Hence,  
\begin{equation}
\| x_k - \phi(x_k) \|^{\sharp}_{\chi} < \epsilon/3
\label{equ:EAmenableI}
\end{equation}
for $1 \leq k \leq m$.

By Proposition \ref{prop:VNAUniqueness}, 
let $u \in \Mul$ be a unitary
such that
$$\| \phi(x_k) - u \psi(x_k) u^* \|^{\sharp}_{\chi} < \epsilon/3$$
for $1 \leq k \leq m$.
From this and (\ref{equ:EAmenableI}),
$$\| x_k - u \psi(x_k) u^* \|^{\sharp}_{\chi} < \epsilon$$
for $1 \leq k \leq m$.

     By \cite{GromovMilman}, the unitary group of  
$u (1_{\N} \otimes \mathbb{B}(\h)) u^* \cong \mathbb{B}(\h)$, given
the $\sigma$-strong* topology, is extremely amenable.
Hence, by Lemma \ref{lem:pointwise}, since
$\epsilon$ and $x_1, x_2, \cdots, x_m$ are arbitrary (so we can construct
an appropriate locally dense net),
$U(\Mul)$ is extremely amenable.
\end{proof}





\section{The Kirchberg property}

    If $G$ is a topological group and $\Gamma$ is a discrete group, then
let $Rep(\Gamma, G)$ be the collection of group homomorphisms from
$\Gamma$ to $G$.  $Rep( \Gamma, G)$ is naturally a subset of the
(topological) infinite product $G^{\Gamma}$ (which has the product topology).
Give $Rep(\Gamma, G)$
the restriction of the topology from $G^{\Gamma}$.  We call this topology
on $Rep(\Gamma, G)$ the \emph{pointwise-$G$ topology}.

\begin{df}  A topological group $G$ is said to have the
\emph{Kirchberg property} if 
every homomorphism from
$\Fr_{\infty} \times \Fr_{\infty}$ into $G$ can be approximated, in the
pointwise-$G$ topology,
by homomorphisms with precompact image.
\label{df:KirchbergProperty}
\end{df}

The Kirchberg property seems to be a type of amenability property;  in
particular, it implies amenability.
    In \cite{Pestov}, Pestov and Uspenskij showed that the isometry group of the Urysohn
metric space has the Kirchberg property.
In the same paper, Pestov and Uspenskij
also observed the following equivalences:

\begin{prop}   The following two statements are equivalent:
\begin{enumerate}
\item The Connes embedding problem has an affirmative answer.
\item The unitary group $U( \h)$ over a separable infinite dimensional
Hilbert space $\h$,
given the weak* topology, has the Kirchberg property.
\end{enumerate}
\label{prop:Pestov}
\end{prop}

Following this line of thought, we use our uniqueness theorem
(Proposition \ref{prop:VNAUniqueness}) and ideas from the theory of
absorbing extensions to show that the Connes embedding problem
is equivalent to many topological groups having the Kirchberg
property.
Firstly, we fix some notation.  For the multiplier algebra $\Mul(\B)$
of a stable $C^*$-algebra $\B$, let $U( \Mul(\B))$ be the unitary
group of $\Mul(\B)$, given the strict topology.
For a von Neumann algebra $\Mul$, let $U( \Mul )$ be the unitary group
of $\Mul$, given the weak* topology (which, on the unitary group, is
the same as the strong topology, weak topology,
$\sigma$-strong* topology, $\sigma$-weak topology etc.).

\begin{thm}  The following statements are equivalent:
\begin{enumerate}
\item The Connes embedding problem has an affirmative answer.

\item For every unital separable simple nuclear $C^*$-algebra $\A$,
$U( \Mul( \A \otimes \K))$ has the Kirchberg property.

\item There exists a unital separable simple nuclear $C^*$-algebra $\A$
such that $U(\Mul(\A \otimes \K))$ has the Kirchberg property.
\item For every countably decomposable
injective properly infinite von Neumann algebra $\Mul$,  
$U( \Mul )$ has the Kirchberg property.
\item There exists an injective properly infinite
von Neumann algebra $\Mul$ with separable predual 
such that $U( \Mul )$ has the Kirchberg property.
\end{enumerate}
\end{thm}

\begin{proof}

    That (2) implies (1) follows from Proposition \ref{prop:Pestov}.
That (2) implies (3) and that (4) implies (5) are immediate.

     We now prove that (1) implies (2).
Let $\pi : \Fr_{\infty} \times \Fr_{\infty} \rightarrow
U ( \Mul( \A \otimes \K))$ be a group homomorphism.

Recall that the strict topology on $\Mul( \A \otimes \K)$ is generated
by seminorms of the form $\|. \|_b$, $b \in \A \otimes \K$,   where
for all $m \in \Mul(\A \otimes \K)$ and for all $b \in \A \otimes \K$,
$\| m \|_b =_{df} \| mb \| + \| b m \|$.

So let $\epsilon > 0$ and let a finite subset
$\{ x_1, x_2, ..., x_m \} \subseteq \Fr_{\infty} \times
\Fr_{\infty}$ and
a finite subset $\G \subseteq \A \otimes \K$ be given. (We will work
with the seminorms corresponding to elements of $\G$).
Contracting $\epsilon > 0$ if necessary, we may assume that
the elements of $\G$ all have norm less than or equal to one.

Let $S_1, S_2$ be isometries in $\Mul(\A \otimes \K)$ such that
the following hold:
\begin{equation} \label{equ:KirchbergS_1S_2Splitting} \end{equation}
\begin{enumerate}
\item[i.] $S_1 S_1^* + S_2 S_2^* = 1_{\Mul(\A \otimes \K)}$.
\item[ii.] $\| S_1 \pi(x_k) S_1^* - \pi(x_k) \|_b < \epsilon/8$
for $1 \leq k \leq m$.
\item[iii.] $\| S_2 S_2^* \|_b < \epsilon/8$
\end{enumerate}

Let $\C \subset \Mul(\A \otimes \K)$ be the unital C*-subalgebra
that is generated by $\{ \pi(x_1), \pi(x_2), ..., \pi(x_m) \}$.
Let $\psi : \C \rightarrow 1_{\A} \otimes \mathbb{B}(\h) \subset
\Mul(\A \otimes \K)$ be a unital full *-homomorphism.

Let $\phi : \C \rightarrow \Mul(\A \otimes \K)$ be the unital full
*-homomorphism that is given by
$\phi(c) =_{df} S_1 c S_1^* + S_2 \psi(c) S_2^*$.
Note that by (\ref{equ:KirchbergS_1S_2Splitting}), we have that
\begin{equation}
\| \phi(\pi(x_k)) - \pi(x_k) \|_b < \epsilon/4
\label{equ:KirchbergNextStep}
\end{equation}
for $1 \leq k \leq m$ and $b \in \G$.

By \cite{EK},  $\phi$ and $\psi$ are approximately unitarily
equivalent (in the pointwise-norm topology).
Hence, let $u \in \Mul(\A \otimes \K)$ be a unitary such that
$\| \phi(\pi(x_k)) - u \psi(\pi(x_k)) u^* \| <\epsilon/8$
for $1 \leq k \leq m$.
From this and (\ref{equ:KirchbergNextStep}),
\begin{equation}
\| \pi(x_k) - u \psi(\pi(x_k)) u^* \|_b < \epsilon/2
\label{equ:KirchbergAlmostLastStep}
\end{equation}
for $1 \leq k \leq m$ and $b \in \G$.
But by (1) and Proposition \ref{prop:Pestov},
$u(1_{\A} \otimes \mathbb{B}(\h)) u^*$ is assumed to
have the Kirchberg property.
Hence, there is a group homomorphism
$\mu : \Fr_{\infty} \times \Fr_{\infty} \rightarrow
u(1_{\A} \otimes \mathbb{B}(\h)) u^*$ with precompact range such that
$\| \mu(x_k) - u \psi(\pi(x_k)) u^* \|_b < \epsilon/2$
for $1 \leq k \leq m$ and $b \in \G$.
From this and (\ref{equ:KirchbergAlmostLastStep}),
we have that
$$\| \pi(x_k) -  \mu(x_k) \|_b < \epsilon$$
for $1 \leq k \leq m$ and $b \in \G$.
Since $\{ x_1, x_2, ..., x_k \}$, $\epsilon$ and $\pi$ are arbitrary,
$U(\Mul(\A \otimes \K))$ has the Kirchberg property.
This completes the proof that (1) implies (2).

Next, we prove that (3) implies (1).

 Let $\h$ be a separable
infinite dimensional Hilbert space. By Proposition \ref{prop:Pestov},
to prove (1), it suffices to prove that $U(\h)$, given the strong topology,
has the Kirchberg property.

     Since $\A$ is a unital separable simple $C^*$-algebra, we may
assume that $\A \otimes \K$ is a $C^*$-subalgebra of $\mathbb{B}(\h)$
such that
\begin{enumerate}
\item[i.] $\A \otimes \K$ acts nondegenerately on $\h$;
\item[ii.]
$\Mul(\A \otimes \K) = \{ m \in \mathbb{B}(\h) : mb, bm \in \A \otimes \K,
\forall b \in \A \otimes \K \}$; and
\item[iii.]  the strict topology on $\Mul(\A \otimes \K)$ is stronger than
the strong operator topology from $\mathbb{B}(\h)$.
\end{enumerate}
(Note that this implies
that $1_{\Mul(\A \otimes \K)} = 1_{\mathbb{B}( \h)}$.) 

   The rest of the argument is the same as that of (1) implies (2),
except that we reverse $\mathbb{B}(\h)$ (or
$1_{\A} \otimes \mathbb{B}(\h)$) and $\Mul(\A \otimes \K)$.

The arguments that (1) implies (4) and that (5) implies (1) is the same
as that of (1) implies (2) and (3) implies (1) respectively. 
We note that since $\Mul$ is properly infinite,  
$\Mul \cong \Mul \otimes \mathbb{B}(\h)$.  And also, in place of the
uniqueness results from \cite{EK}, we need to use
Proposition \ref{prop:VNAUniqueness}.
\end{proof}


\begin{thebibliography}{00}



\bibitem{ArvesonDuke} W. Arveson, \textit{Notes on extensions of C*-algebras,}
Duke Math. J., \textbf{44} (1977) no. 2, 329--355.



\bibitem{Bratteli} O. Bratteli, \textit{Inductive limits of finite-dimensional
C*-algebras,} Trans. Amer. Math. Soc., \textbf{171} (1972)
195--234.





\bibitem{CGNN}  A. Ciuperca, T. Giordano, P. W. Ng and Z. Niu,
\textit{Amenability and uniqueness,} preprint.  	arXiv:1207.6741v1 [math.OA]

\bibitem{delaHarpe} P. de la Harpe, {\it Moyennabilit\'e du groupe unitaire et
propri\'et\'e P de Schwartz des alg\`ebres de von Neumann} \;
Alg\`ebres d'op\'erateurs,
Lecture Notes in Mathematics Vol. 725 (Springer-Verlag, New York, 1979)
pp. 220--227






\bibitem{DavidsonBook} K. R. Davidson,
\textit{C*-algebras by example,} Fields Institute Monographs,\textbf{6}. 

\bibitem{EK} G. A. Elliott and D. Kucerovsky; An abstract
Brown-Douglas-Fillmore absorption theorem; Pacific J. of Math., \textbf{3}, (2001),
1-25

\bibitem{FutamuraKataokaKishimoto} 
H. Futamura, N. Kataoka and A. Kishimoto,
\textit{Type III representations and automorphisms of some separable 
nuclear C*-algebras,} J. Funct. Anal., \textbf{197}, (2003), 
no. 2, 560--575. 

\bibitem{GP2} T. Giordano and V. Pestov, 
\emph{Some extremely amenable groups related to operator algebras and
ergodic theory,} J. Inst. Math. Jussieu, \textbf{6} (2007) no. 2,
279--315. 

\bibitem{GromovMilman} M. Gromov and V. D. Milman, \textit{A topological
application of the isoperimetric inequality},  Amer. J. Math., \textbf{105}, (1983),
no. 4, 843-854

\bibitem{KadisonVolume2} R. V. Kadison and J. R. Ringrose, \textit{Fundamentals
of the theory of operator algebras.  Vol. II. Advanced theory}
Graduate studies in mathematics;16;1997; Corrected reprint of the
1986 original. American Mathematical Society, Providence, RI

\bibitem{KishimotoFlows} A. Kishimoto, \textit{Approximately inner flows
on separable C*-algebras.  Dedicated to Professor Huzihiro Araki on
the occasion of his 70th birthday.} Rev. Math. Phys., \textbf{14} 
(2002) no. 7--8, 649--673.  

\bibitem{KishimotoInternat}  A. Kishimoto, \textit{The representations
and endomorphisms of a separable nuclear C*-algebra,}
Internat. J. Math., \textbf{14}, (2003), no. 3, 313-326. 

\bibitem{KishimotoOzawaSakai}  A. Kishimoto, N. Ozawa and S. Sakai,
\textit{Homogeneity of the pure state space of a separable C*-algebra,}
Canad. Math. Bull., \textbf{46}, (2003), no. 3, 365--372. 











\bibitem{KucerovskyNg} D. Kucerovsky and P. W. Ng,
\textit{Nuclearity through absorbing extensions,} C. R. Math. 
Acad. Sci. Soc. R. Can., \textbf{32}, (2010), no. 4, 106--119.  



\bibitem{Pestov} V. G. Pestov and V. V. Uspenskij,\textit{Representations of
	residually finite groups by isometries of the Urysohn space,}
J. Ramanujan Math. Soc., \textbf{21},(2006), no. 2, 189-203

\bibitem{Powers} R. T. Powers, \textit{Representations of uniformly 
hyperfinite algebras and their associated von Neumann rings,}
Ann. Math., \textbf{86}, (1967), 138--171.  

\bibitem{Veech} W.A. Veech, \textit{Topological dynamics,}
Bull. Amer. Math. Soc., \textbf{83}, (1977), 775--830.  





\end{thebibliography}
\end{document}